\documentclass[11pt]{amsart}

\usepackage{amsthm}
\usepackage{amsmath}
\usepackage[textsize=tiny]{todonotes}
\usepackage[margin=1in]{geometry}
\usepackage{amssymb,verbatim,enumitem}
\usepackage{tikz}
\usepackage{tikz-cd}
\usepackage{multirow}

\edef\restoreparindent{\parindent=\the\parindent\relax}
\usepackage[skip=2pt]{parskip}
\restoreparindent

\theoremstyle{plain}
\newtheorem{theorem}{Theorem}[section]          
     
\newtheorem{lemma}[theorem]{Lemma}              
\newtheorem{proposition}[theorem]{Proposition}

\newtheorem{question}{Question}[section]

\newtheorem{maintheorem}{Theorem}

\newtheorem{maincorollary}[maintheorem]{Corollary}
\theoremstyle{definition}
\newtheorem{definition}[theorem]{Definition}

\newtheorem{remark}[theorem]{Remark}
\newtheorem{example}[theorem]{Example}

\numberwithin{equation}{section}
\numberwithin{table}{section}

\newcommand{\Z}{\mathbb{Z}}
\newcommand{\R}{\mathbb{R}}\newcommand{\C}{\mathbb{C}}\newcommand{\HH}{\mathbb{H}}

\newcommand{\RR}{\mathbb{R}}

\newcommand{\SO}{\mathsf{SO}}

\newcommand{\Sp}{\mathsf{Sp}}
\newcommand{\Spin}{\mathsf{Spin}}
\newcommand{\SU}{\mathsf{SU}}

\newcommand{\fsp}{\mathfrak{sp}}

\newcommand{\CP}{\mathbb{C}\mathrm{P}}

\newcommand{\Gr}{\mathrm{Gr}}

\newcommand{\cH}{\mathcal{H}}

\newcommand{\cP}{\mathcal{P}}

\DeclareMathOperator{\Ric}{Ric}
\DeclareMathOperator{\Span}{span}
\newcommand{\pr}[1]{\Ric_{#1}>0}

\newcommand{\of}[1]{\left(#1\right)}

\newcommand{\im}{\mathrm{Im}}

\mathcode`l="8000
\begingroup
\makeatletter
\lccode`\~=`\l
\DeclareMathSymbol{\lsb@l}{\mathalpha}{letters}{`l}
\lowercase{\gdef~{\ifnum\the\mathgroup=\m@ne \ell \else \lsb@l \fi}}%
\endgroup

\renewcommand{\Im}{\mathrm{Im}}

\usepackage{hyperref}

\newcommand*{\defeq}{\mathrel{\vcenter{\baselineskip0.5ex \lineskiplimit0pt
			\hbox{\scriptsize.}\hbox{\scriptsize.}}}%
	=}

\title[Positive intermediate Ricci curvature on cohomogeneity one manifolds]{Positive intermediate Ricci curvature on cohomogeneity one manifolds in low dimensions}
\author{Elahe Khalili Samani}
\address{Clark University, Worcester, MA USA}
\email{EKhaliliSamani@clarku.edu}
\author{Lawrence Mouill\'e}
\address{Trinity University, San Antonio, TX USA}
\email{lmouille@trinity.edu}
\date{\today}

\begin{document}

\begin{abstract}
	We explore existence of invariant metrics with positive intermediate Ricci curvature on closed, low-dimensional cohomogeneity one manifolds. 
	For a certain cohomogeneity one $\Spin(4)$-action on $S^3 \times \CP^2$, we construct an invariant metric with positive $4^\mathrm{th}$-intermediate Ricci curvature 
	and show it cannot admit an invariant metric with positive $3^\mathrm{rd}$-intermediate Ricci curvature. 
	We further establish similar symmetry obstructions to positive curvature for $S^3 \times S^3$, $S^3 \times S^4$, and several families of cohomogeneity one manifolds.
\end{abstract}

\maketitle

\section{Introduction}

The study of Riemannian manifolds with positive sectional curvature has been a central problem in Riemannian geometry since the creation of the subject.
Although several topological obstructions to positive sectional curvature are known, the list of known examples remains very limited, with the rank one symmetric spaces remaining the only known simply connected examples in dimensions larger than $24$. An area of focus receiving increased attention lately has been the investigation of {\it positive $k^{th}$-intermediate Ricci curvature} 
($\pr{k}$), a condition that interpolates between positive sectional curvature ($k=1$) and positive Ricci curvature ($k=n-1$) on an $n$-manifold (see Definition \ref{def:Rick}). 
One motivating goal in this area is to extend results known for positive sectional curvature, like the Gromoll-Meyer Theorem extended by Shen \cite{Shen93}, the Synge Theorem by Wilhelm \cite{Wilhelm97}, the Grove-Searle Theorem by Kennard and the second author \cite{Mouille22b,KennardMouille24}, and a series of results by Guijarro and Wilhelm \cite{GuijarroWilhelm18,GuijarroWilhelm20,GuijarroWilhelm22}.
Another key goal is to construct new examples that do not admit positive sectional curvature, like the homogeneous examples given by Dom\'inguez-V\'azquez, Gonz\'alez-\'Alvaro, DeVito, Rodr\'iguez-V\'azquez, and the second author \cite{DVGAM23,DMDApreprint} and several constructions by Reiser and Wraith \cite{ReiserWraith22preprint2,ReiserWraith23,ReiserWraith23preprint2,ReiserWraith23preprint1}.

%For simply connected homogeneous spaces, we have a complete classification for positive sectional curvature (see \cite{WilkingZiller18} and references therein for a complete history).
%On the other hand, every homogeneous space with finite fundamental group admits an invariant metric with positive Ricci curvature \cite{Nash79,Berestovskii95}. 
%In \cite{DVGAM23}, Dom\'inguez-V\'azquez, Gonz\'alez-\'Alvaro, and the second author applied Cheeger deformations to construct infinitely many homogeneous examples with $\pr{k}$. 
%Following a similar approach, DeVito, Dom\'inguez-V\'azquez, Gonz\'alez-\'Alvaro, and Rodr\'iguez-V\'azquez expanded the list to include new examples with $\pr{2}$ in \cite{DMDApreprint}. 
%However, a classification of all homogeneous spaces with $\pr{k}$ is still far from reach.

Focusing on closed, simply connected cohomogeneity one manifolds, a classification of those with positive sectional curvature was obtained, except in dimension seven 
in which there are two families of possible candidates, by Verdiani, Grove, Wilking, and Ziller \cite{Verdiani02,Verdiani04,GroveWilkingZiller08,VerdianiZiller14}. (For an extension of this classification to quasipositive curvature, see \cite{Wulle24}.) On the other hand, Grove and Ziller showed that any cohomogeneity one manifold with finite fundamental group admits an invariant metric 
with positive Ricci curvature \cite{GroveZiller02}. 
%the list of known examples with positive $\Ric_2$ is limited to $S^3\times S^3$, $S^3\times S^2$, and $S^2\times S^2$. 
This naturally leads to the question: What is the smallest $k$ for which a given cohomogeneity one manifold admits an invariant metric with $\pr{k}$? 

In this paper, we address this question in low dimensions. Our first result is the following:

\begin{maintheorem}\label{main-thm:S3CP2}
%	Let $\phi: S^3 \to \SO(3)$ denote the standard covering map. 
	$S^3\times\C\mathrm{P}^2$ admits a metric with $\Ric_4 > 0$ that is invariant under the cohomogeneity one 
	$(S^3 \times S^3)$-action with the group diagram
	\[
		H=\langle(i,i)\rangle \subset \{(e^{i\theta},e^{i\theta})\}, \{(e^{j\theta},e^{j\theta})\}\cdot H\subset S^3\times S^3.
	\]
	Furthermore, it does not admit a metric with $\pr{3}$ invariant under this action.    
%	Let $\phi: S^3 \to \SO(3)$ denote the standard covering map. 
%	Then $S^3\times\C\mathrm{P}^2$ admits a metric with $\Ric_4 > 0$ that is invariant under the cohomogeneity one 
%	$(S^3 \times S^3)$-action given by\vspace{-0.2cm}
%	\[
%		(a,b) * (p,[z]) = (apb^{-1},[\phi(b) z]),
%	\]
%	while it does not admit any invariant metric with $\pr{3}$.    
\end{maintheorem}

For a description of this action, see Equation \eqref{eq:s3xcp2action} below.
The metric we construct is a Cheeger deformation by the action of an $S^3$ subgroup; for details see the exposition preceding Proposition \ref{P:s3xcp2}.
The Hopf action on the  first factor of $S^3\times\C\mathrm{P}^2$ is by isometries of this metric, and thus the Gray-O'Neill formula (Lemma \ref{lem:oneill}) applied to the quotient map $S^3\times\C\mathrm{P}^2 \to S^2 \times \CP^2$ gives the following:

\begin{maincorollary}\label{cor:s2xcp2}
	$S^2 \times \CP^2$ admits a metric with $\pr{4}$ that is invariant under the diagonal action of $\SO(3)$, which is of cohomogeneity three.
\end{maincorollary}

%Compare Corollary \ref*{cor:s2xcp2} to \cite[Theorem C]{DVGAM23} which states that a normal homogeneous space $M^n$ has only $\pr{n-1}$ (and does not have $\pr{n-2}$) if and only if it locally splits off a factor isometric to a round 2-sphere.
%This situation is similar to $S^2 \times S^2$, which is known to admit a cohomogeneity one metric with $\pr{2}$; see Example \ref{ex:s3xs3} below.

%\todo[inline]{What is the isometry group of our metric in Theorem A? Or the identity component of the isometry group? If it doesn't act transitively, ask if there exists a homogeneous metric with $\pr{k}$ for $k\leq 4$.}

Since our metrics in Theorem \ref{main-thm:S3CP2} and Corollary \ref{cor:s2xcp2} are of cohomogeneity one and three, respectively, it is natural to ask the following:

\begin{question}
	Does $S^3 \times \CP^2$ or $S^2 \times \CP^2$ admit a homogeneous metric with $\pr{4}$?
\end{question}

In \cite{VerdianiZiller14}, Verdiani and Ziller applied an argument involving transverse families of Jacobi fields to rule out existence of invariant metrics with $\pr{2}$ on a family of $7$-dimensional cohomogeneity one manifolds (see Example \ref{ex:PD7}). 
We follow similar techniques to prove the last statement in Theorem \ref*{main-thm:S3CP2}.
Taking this method further, using Hoelscher's classification of low-dimensional cohomogeneity one manifolds in \cite{Hoelscher10}, we establish our second result:

\begin{maintheorem}\label{main-thm:obstruction}
The following cohomogeneity one manifolds do not admit invariant metrics with $\pr{2}$:
\begin{enumerate}
\item  The Brieskorn varieties $M_d^5$ with the group diagram
		\[
		\{(1,1)\}\subset \{(e^{i\theta}, 1)\}, \{(e^{jd\theta}, e^{i\theta})\}\subset S^3\times S^1,
		\]  
		where $d$ is even.
\item The family of cohomogeneity one manifolds $Q_A^5$ with the group diagram
		\[
		\{(1,1)\}\subset\{(e^{ip\theta},e^{i\theta})\}, \{(e^{ip\theta},e^{i\theta})\}\subset S^3\times S^1.
		\]
\item The family of cohomogeneity one manifolds $N_C^7$ with the group diagram 
		\[
		H\subset\{(e^{ip\theta},e^{iq\theta})\}, S^3\times\Z_n\subset S^3\times S^3,
		\]		
        where $(q,n)=1$, $\Z_n\cong H\subset\{(e^{ip\theta},e^{iq\theta})\}$, assuming $n = 1$ or $2$.
\end{enumerate}
\vspace{10pt}
\noindent
Moreover, the following cohomogeneity one manifolds do not admit invariant metrics with $\pr{3}$:
\begin{enumerate}[resume]
\item The cohomogeneity one manifold $S^3\times S^3$ with the group diagram
		\[
		\{e^{i\theta}\}\times 1\subset S^3\times 1, S^3\times 1\subset S^3\times S^3.
		\]
\item The family of cohomogeneity one manifolds $P_A^7$ with the group diagram
		\[
		H=\langle(i,i)\rangle\subset\{(e^{ip_-\theta},e^{iq_-\theta})\}, \{(e^{jp_+\theta},e^{jq_+\theta})\}\cdot H\subset S^3\times S^3,
		\]
where $p_-, q_-\equiv 1\mod 4$. %, and moreover, one of $p_+$ and $q_+$ is even and the other one is odd.	
\item The family of cohomogeneity one manifolds $Q_C^7$ with the group diagram
		\[
		S^1\times 1\times 1\subset\{(e^{i\phi},e^{ib\theta},e^{i\theta})\}, S^3\times 1\times 1\subset S^3\times S^3\times S^1.
		\]
\item The cohomogeneity one manifold $S^4 \times S^3$ with the group diagram
		\[
		\{(1,1)\}\subset S^3\times 1, S^3\times 1\subset S^3\times S^3.
		\]		
\end{enumerate}
\end{maintheorem}

We now discuss the current state of the art for low-dimensional cohomogeneity one metrics with positive intermediate Ricci curvature.
In dimensions seven and six, the only obstructions are those in Theorem \ref{main-thm:S3CP2}, Theorem \ref{main-thm:obstruction}, \cite[Theorem D]{VerdianiZiller14}, and those concerning positive sectional curvature (see \cite{Verdiani02,Verdiani04,GroveWilkingZiller08,VerdianiZiller14}).
While the metrics constructed in Theorem \ref{main-thm:S3CP2} and Corollary \ref{cor:s2xcp2} add to the list of known examples, there is still much left to discover in these dimensions.
%For previously known examples of metrics with positive intermediate Ricci curvature, see \cite{DVGAM23,ReiserWraith22preprint2,ReiserWraith23,ReiserWraith23preprint1}.

In dimension five, Hoelscher showed the only closed, simply connected, cohomogeneity one manifolds are diffeomorphic to the sphere $S^5$, the Wu manifold $\SU(3)/\SO(3)$, or either of the two $S^3$-bundles over $S^2$ \cite[Theorem C]{Hoelscher10}.
The standard metric on $S^5$ has positive sectional curvature, any normal homogeneous metric on $\SU(3)/\SO(3)$ has $\pr{3}$ (observed independently by Dom\'inguez-V\'azquez, Gonz\'alez-\'Alvaro, and the second author in \cite{DVGAM23} and Amann, Quast, and Zarei in \cite{AmannQuastZarei20preprint}), and the trivial bundle $S^3 \times S^2$ admits a homogeneous metric with $\pr{2}$ (see Example \ref{ex:s3xs3}).
However, the following are still open:
\begin{question}
	Does $\SU(3)/\SO(3)$ admit a metric with $\pr{2}$?
\end{question}
\begin{question}
	Does the non-trivial $S^3$-bundle over $S^2$ admit a metric with $\pr{3}$?
\end{question}
%Notice that in Theorem \ref*{main-thm:obstruction}, we rule out $(S^3 \times S^1)$-invariant metrics with $\pr{2}$ on the $5$-manifolds in the families $M_d^5$ and $Q_A^5$, some of which are diffeomorphic to $\SU(3)/\SO(3)$ or the non-trivial $S^3$-bundle over $S^2$.

%Notice that in Theorem \ref*{main-thm:obstruction}, we rule out $(S^3 \times S^1)$-invariant metrics with $\pr{2}$ on the $5$-manifolds $M_d^5$ and $Q_A^5$.
%According to \cite[Theorem C]{Hoelscher10}, the only compact, simply connected cohomogeneity one manifolds in dimension five are diffeomorphic to the sphere $S^5$, the Wu manifold $\SU(3)/\SO(3)$, or either of the two $S^3$-bundles over $S^2$.
%Furthermore, each of these manifolds admits a cohomogeneity one action by $S^3 \times S^1$.
%The round metric on $S^5$ has positive sectional curvature and is invariant under such an action.
%Also, $S^3 \times S^2$ is known to admit an metric with $\pr{2}$ that is invariant under a cohomogeneity one $(S^3 \times S^1)$-action; see Example \ref{ex:s3xs3} below.
%Since it is currently unknown whether the Wu manifold or the non-trivial $S^3$-bundle over $S^2$ admit metrics with $\pr{2}$, we pose the following:
%
%\begin{question}
%	Does $\SU(3)/\SO(3)$ or the non-trivial $S^3$-bundle over $S^2$ admit a metric with $\pr{2}$ that is invariant under a cohomogeneity one $(S^3 \times S^1)$-action?
%\end{question}

In dimension four, Parker showed the only closed, simply connected, cohomogeneity one manifolds are diffeomorphic to the sphere $S^4$, the complex projective space $\CP^2$, the product of spheres $S^2 \times S^2$, or the connected sum of complex projective spaces with opposite orientations $\CP^2 \# \overline{\CP^2}$ \cite{Parker86} (see also \cite{Hoelscher07thesis}).
While the first two admit metrics with positive sectional curvature and $S^2 \times S^2$ admits a cohomogeneity one metric with $\pr{2}$ (see Example \ref{ex:s3xs3}), the following is still open:

\begin{question}
	Does $\CP^2 \# \overline{\CP^2}$ admit a metric with $\pr{2}$?
\end{question}

%\todo[inline]{Does this include Wu manifold or non-trivial $S^3$-bundle over $S^2$?}

%\todo[inline]{Jason says he and Lee proved the even-dimensional CROSSes have primitive group diagrams, and Jason later proved the odd-dimensional case. Our main construction has non-primitive group diagram. Which of the spaces that we rule out have non-primitive group diagram? Should we ask a question saying if $\pr{k}$, and maybe dimension is low, then group diagram must be primitive? Being that our example if non-primitive, do we get a submersion to some space like $S^3$?}

Finally, in dimensions three and two, the only closed, simply connected, cohomogeneity one manifolds are the spheres $S^3$ and $S^2$ (see \cite{Hoelscher07thesis}), which of course admit metrics with positive sectional curvature.

This paper is organized as follows: 
In Section \ref{S:preliminaries}, we collect tools and observations used throughout the paper. 
Section \ref{S:constructions} is devoted to the construction of the metric with $\pr{4}$ on $S^3 \times \CP^2$ described in Theorem \ref{main-thm:S3CP2}.
The proof of Theorem \ref{main-thm:obstruction},
as well as that of the second claim in Theorem \ref{main-thm:S3CP2}, is given in Section \ref{S:obstructions}.
Specifically, the proof of Theorem \ref{main-thm:S3CP2} is given in Propositions \ref{P:s3xcp2} and \ref{P:P_A^7}.
Throughout the paper, we follow the notational conventions established by Hoelscher in \cite{Hoelscher10}.

\subsection*{Acknowledgments}
	This work was initiated when the second author visited Clark University; he thanks the Department of Mathematics for their hospitality.
	The authors would like to thank Lee Kennard and Jason DeVito for their comments on a preliminary version of this article.
	They also thank an anonymous referee for helpful suggestions.
	The first author was supported by the AMS-Simons Travel Grant, while the second author was supported by NSF Award DMS-2202826 and the AMS-Simons Research Enhancement Grant for PUI Faculty.

\section{Preliminaries}\label{S:preliminaries}

\subsection{Intermediate Ricci curvature}\label{SS:Ric_k}
	
	Let $(M,g)$ be a Riemannian manifold with associated Levi-Civita connection $\nabla$ and Riemannian curvature tensor $R$. 
	Given $p\in M$ and a non-zero vector $x\in T_pM$, the directional curvature operator $R_x:\Span\{x\}^\perp\to\Span\{x\}^\perp$ (also known as the Jacobi operator 
	or the tidal force operator) is the self-adjoint operator given by
	\[
		R_x(y) = R(y,x)x.
	\]

	\begin{definition}\label{def:Rick}
		We say that $(M,g)$ has {\it{positive $k^{th}$-intermediate Ricci curvature}} ($\Ric_k>0$) if either of the following equivalent conditions holds:
		\begin{enumerate}
			\item For every $p\in M$ and every non-zero $x\in T_pM$, the directional curvature operator $R_x$ is $k$-positive, that is, the sum of any $k$ eigenvalues 
			of $R_x$ is positive.
			\item For $p\in M$ and orthonormal vectors $x,y_1,\dots,y_k\in T_pM$, one has $\displaystyle\sum_{i=1}^k \sec(x,y_i) > 0$.
		\end{enumerate}
	\end{definition}
	
	For a list of works related to intermediate Ricci curvature, see \cite{Mouillewebsite}.
	Riemannian product provide a natural method for constructing new examples.
	An elementary calculation justifies the following:
	
	\begin{lemma}\label{lem:product}
		If $(M_1,g_1)$ has $\pr{k_1}$ and $(M_2,g_2)$ has $\pr{k_2}$, then the Riemannian product $(M_1 \times M_2,g_1 \oplus g_2)$ has $\pr{k}$, 
		where $k = \max\{k_1 + \dim M_2, k_2 + \dim M_1\}$.
	\end{lemma}
	
	The next result is a simple consequence of the Gray-O'Neill curvature formulas for Riemannian submersions \cite{Gray67,ONeill66}:
	
	\begin{lemma}\label{lem:oneill}
		Let $\pi: (M,g_M) \to (B,g_B)$ be a Riemannian submersion.
		Denote by $\cH$ the horizontal distribution for $\pi$, meaning $\cH_{p}$ is the orthogonal complement of $\ker d\pi_p$ with respect to $g_M$ for each $p \in M$.
		If $\sum_{i=1}^k \sec(x,y_i) > 0$ for all orthonormal vectors $x,y_1,\dots,y_k$ contained in $\cH$, then $(B,g_B)$ has $\Ric_k > 0$.
	\end{lemma}
	
	We now illustrate Lemmas \ref*{lem:product} and \ref*{lem:oneill} with the following:
	
	\begin{example}\label{ex:s3xs3}
		By Lemma \ref*{lem:product}, the product metric on $S^3 \times S^3$ has $\pr{4}$.
		However, $S^3 \times S^3$ admits a metric with $\Ric_2 > 0$ that is invariant under the action of $S^3 \times S^3 \times S^3$ given by
		\[
			(a,b,c)\cdot(p,q) = (apb^{-1},aqc^{-1}).
		\]
		This metric is a Cheeger deformation of the product action via the action of the subgroup $S^3 \times \{1\} \times \{1\}$.
		(See Section \ref{SS:Cheeger} for the definition of Cheeger deformation and Example \ref{ex:s3xs3revisited} for why this metric has $\pr{2}$.)
		Letting $S^1$ denote the complex circle $\{e^{i \theta}\}$ in $S^3$, the quotient of $S^3\times S^3$ by the free action of the subgroup $\{1\} \times \{1\} \times S^1$ 
		is diffeomorphic to $S^3 \times S^2$, and it follows from Lemma \ref*{lem:oneill} that this space inherits a metric with $\Ric_2 > 0$. 
		This metric is invariant under the induced transitive action by $S^3 \times S^3 \times \{1\}$, and hence it is invariant under the cohomogeneity one action by $S^3 \times S^1 \times \{1\}$. 
		On the other hand, the quotient of $S^3\times S^3$ by the subgroup 
		$\{1\} \times S^1 \times S^1$ provides a metric with $\pr{2}$ on $S^2 \times S^2$. This metric is invariant under the induced cohomogeneity one action 
		by $S^3 \times \{1\} \times \{1\}$, for which the associated effective action is the diagonal action of $\SO(3)$ on $S^2 \times S^2$.
	\end{example}

\subsection{Cohomogeneity one manifolds}\label{SS:cohomogeneity one}

	Suppose $M$ is a closed Riemannian manifold with a cohomogeneity one action by a Lie group $G$ such that $M/G=[0,L]$. It is well-known that any such cohomogeneity one   
	manifold has a group diagram $H\subset K_-, K_+ \subset G$ obtained as follows. Let $\pi:M\to M/G$ denote the quotient map and fix $p\in\pi^{-1}(t_0)$ for some $t_0\in(0,L)$. 
	Fixing a $G$-invariant metric $g$ on $M$ such that $M/G$ is isometric to $[0,L]$ and letting $\gamma:[0,L]\to M$ be the unique minimal horizontal geodesic 
	with $\gamma(t_0)=p$, the groups $H$, $K_-$, and $K_+$ are given by $G_{\gamma(t_0)}$, $G_{\gamma(0)}$, and $G_{\gamma(L)}$, respectively. 
	Moreover, changing the choice of the $G$-invariant metric $g$ or horizontal geodesic gives rise to a different group diagram, and all possible group diagrams are described in the following 
	(see \cite[p. 44]{GroveWilkingZiller08}):
	
	\begin{lemma}\label{L:diagram}
		Let $M$ be a cohomogeneity one manifold with the group diagram $H \subset  K_-,K_+  \subset G$. A cohomogeneity one $G$-manifold $N$ is equivariantly diffeomorphic 
		to $M$ if and only if its group diagram is of the form (up to switching $K_-$ and $K_+$):
		\[
			aHa^{-1}\subset anK_-n^{-1}a^{-1}, aK_+a^{-1} \subset G,
		\]
		where $a\in G$ and $n\in N(H)_0$, the identity component of the normalizer of $H$ in $G$.
	\end{lemma}

	Throughout the paper, a cohomogeneity one manifold with the group diagram 
	\begin{equation}\label{eq:equivalent}
	H\subset K_-,K_+  \subset G
	\end{equation} 
	refers to any cohomogeneity one manifold whose group diagram is equivalent to \eqref{eq:equivalent} in the sense of Lemma \ref{L:diagram}.
	
	\begin{remark}\label{R:Weyl}
		The geodesic $\gamma$ discussed above can be extended to a geodesic $\gamma:\R\to M$, and $\gamma(\R)$ forms a section for the action of $G$ on $M$. 
		In particular, the action of $G$ on $M$ is polar with the Weyl group $W$ generated by unique involutions $w_\pm \in (N(H) \cap K_\pm)/H$.
		%elements $w_-$ and $w_+$, where $w_-\in K_-$ and $w_+\in K_+$ are unique elements (mod $H$) satisfying $w_-, w_+\notin H$ but $w_-^2, w_+^2\in H$.
		Furthermore, the Weyl group $W$ is finite if and only if the geodesic $\gamma$ is closed, in which case, the length of $\gamma$ is $kL$, where $k$ is the order of $W$.
		Note that the singular isotropy groups along $\gamma$ 
		are of the form $wK_{\pm}w^{-1}$ for some $w\in N(H)$ representing an element of $W$. More precisely, conjugation by a representative of $w_-$ takes $G_{\gamma(kL)}$ 
		to $G_{\gamma(-kL)}$, and conjugation by a representative of $w_+$ takes $G_{\gamma(kL)}$ to $G_{\gamma(2L-kL)}$ (see \cite[p. 368]{VerdianiZiller14}). 
		Thus given representatives $u_\pm$ respectively for $w_\pm$, we have $G_{\gamma(2L+kL)} = (u_+ u_-) G_{\gamma(kL)} (u_+ u_-)^{-1}$.
	\end{remark}

\subsection{Curvature restrictions on Jacobi fields}\label{SS:Jacobi}
	
	Here, we outline how positive intermediate Ricci curvature controls the behavior of action fields on a cohomogeneity one manifold.
	
	Suppose $(M^n,g)$ is a compact, cohomogeneity one $G$-manifold with $M/G=[0,L]$. 
	Let $\gamma:[0,L] \to M$ be a unit speed horizontal geodesic 
	that begins at a singular orbit $N = G(\gamma(0))$. 
	Consider the $(n-1)$-dimensional vector space $\Lambda$ of Jacobi fields along $\gamma$ generated by the $G$-action. 
	For all $t\in(0,L)$, we have $\Span\{\gamma'(t)\}^\perp = \{J(t):J\in \Lambda\}$, and the Riccati operator $S_t:\Span\{\gamma'(t)\}^\perp \to \Span\{\gamma'(t)\}^\perp$ 
	is defined by $S_t(J(t)) = J'(t)$. 
	For $t=0$ or $L$, we have $\Span\{\gamma'(t)\}^\perp = \{J(t):J\in \Lambda\} \oplus \{J'(t) : J \in \Lambda, J(t)=0\}$ 
	(see \cite[Section 1.7]{GromollWalschap09}). 
	Moreover, by Lemma 1.3 in \cite{GuijarroWilhelm22}, the equation $S_t(J(t)) = J'(t)$ provides a well-defined Riccati operator 
	for these values of $t$ as well. 
	Since $\Lambda$ consists of variational fields for the variation of $\gamma$ by geodesics that leave $N$ orthogonally at $t=0$, 
	it follows that the Riccati operator $S_t$ is self-adjoint for all $t$ (see Lemma 1.4 in \cite{GuijarroWilhelm22}). 
	$\Lambda$ is an example of a {\it Lagrangian} family of normal Jacobi fields.
	In other words, with respect to the symplectic form $\omega(J_1,J_2) = g(J_1',J_2) - g(J_1,J_2')$ defined on the $(2n-2)$-dimensional space of all Jacobi fields orthogonal to $\gamma$, $\Lambda$ is equal to its own symplectic complement.
	
	Using Wilking's Transverse Jacobi Equation in \cite[Theorem 9]{Wilking07}, Gumaer and Wilhelm \cite{GumaerWilhelm14} studied consequences of positive intermediate Ricci  
	curvature on Lagrangian families of normal Jacobi fields (see Verdiani and Ziller \cite{VerdianiZiller14} for related independent results).
	A direct consequence of applying \cite[Theorem C]{GumaerWilhelm14} (cf. \cite[Theorem B]{VerdianiZiller14}) to the Lagrangian family $\Lambda$ defined above is the following, 
	which is the key obstruction we use in this paper:
	
	\begin{lemma}\label{L:obstruction}
		Suppose $(M^n,g)$ is a compact, cohomogeneity one $G$-manifold. Let $\gamma$ be a unit speed horizontal geodesic, and let $\Lambda$ be the family of Jacobi fields 
		along $\gamma$ generated by the $G$-action. If\vspace{-0.1cm}
		\[
			\dim(\Span\{J\in\Lambda\mid J(t)=0~{\text{for some}}~t\}) = n-k-1,
		\]
		then the metric $g$ does not have $\Ric_k>0$.
		In other words, there is an orthonormal set of vectors $\{x, y_1, \dots ,y_k\}$ for which $\sec(x,y_1) + \dots + \sec(x,y_k) \leq 0$.
	\end{lemma}
	
	As the following example shows, when applying Lemma \ref*{L:obstruction} to prove that a cohomogeneity one manifold cannot admit an invariant metric with $\Ric_k > 0$, 
	one needs to check all representative group diagrams given in Lemma \ref{L:diagram}.
	
	\begin{example}\label{ex:s2xs2}
		Think of $S^3$ as the set of unit quaternions and consider the action of $S^3$ on $S^2 \subset \Im(\HH)$ given by $p\cdot x=pxp^{-1}$.
		Then the diagonal action of $S^3$ on $S^2 \times S^2$ is of cohomogeneity one. Moreover, the following group diagrams represent the same equivariant diffeomorphism 
		class for this $S^3$-manifold:
		\[
			\{\pm1\} \subset \{e^{i\theta}\}, \{e^{i\theta}\} \subset S^3,
		\]
		\[
			\{\pm1\} \subset \{e^{i\theta}\}, \{e^{j\theta}\} \subset S^3.
		\]
		By applying Lemma \ref{L:obstruction}, one can argue, as in Section \ref{S:obstructions} below, that an $S^3$-invariant Riemannian metric which admits a horizontal geodesic 
		giving rise to the first group diagram does not have $\pr{2}$. In contrast, the metric on $S^2 \times S^2$ described in Example \ref{ex:s3xs3} has $\Ric_2>0$,
		is invariant under the $S^3$-action, and realizes the second group diagram via the geodesic $\gamma(t) = \cos(t) (i,i) + \sin(t) (j,-j)$. 
		Furthermore, this metric admits no horizontal geodesic realizing the first group diagram, nor does it admit a horizontal geodesic realizing any equivalent group diagram 
		whose singular isotropy groups are equal.
	\end{example}
	
\subsection{Cheeger deformation}\label{SS:Cheeger}
	
	In \cite{Cheeger73}, Cheeger developed a method for deforming Riemannian metrics in the presence of isometric group actions, now referred to as a Cheeger deformation.
	In this section, we collect some general facts about Cheeger deformations that will be used in Section \ref{S:constructions} to construct the metric in Theorem \ref{main-thm:S3CP2}.
	
	Let $G$ be a compact group of isometries of $(M,g)$, and let $Q$ be a bi-invariant metric on $G$. Consider the one-parameter family of metrics 
	$\{\bar{g}_l \defeq l^2 Q + g\}_{l > 0}$ on $G \times M$. Note that $G$ acts freely and by isometries on $(G\times M,\bar{g}_l)$ via
	\begin{equation}\label{eq:cheegeraction}
		b * (a,p) = ( a b^{-1} , b\cdot p ),\hspace{0.5cm} p\in M,~a, b\in G.
	\end{equation}
	The quotient of this action is diffeomorphic to $M$, and each metric $\bar{g}_l$ on $G \times M$ induces a metric $g_l$ on $M$ for which the quotient map 
	$q:(G \times M,\bar{g}_l) \to (M,g_l)$ given by $q(a,p)=a\cdot p$ is a Riemannian submersion. Each metric in the family $\{g_l\}_{l>0}$ is referred to as a Cheeger deformation 
	of $g$ (with respect to the $G$-action).
	
	\begin{lemma}\label{lem:isometry}
		Let $(M,g_l)$ denote a Cheeger deformation of $(M,g)$ with respect to a $G$-action. Then the $G$-action is by isometries of $g_l$.
		Furthermore, if a Lie group $H$ acts on $M$ by isometries of $g$ such that the $G$-action and the $H$-action commute, then the $H$-action is also by isometries of $g_l$.
	\end{lemma}
	
	Fixing $p\in M$, consider the $g_l$-symmetric automorphism $C_l:T_p M \to T_p M$ satisfying
	\begin{equation}\label{eq:cheeger}
		g(x,y) = g_l(C_l(x),y)
	\end{equation}
	for all $x, y\in T_pM$. We note $C_l$ preserves the horizontal and the vertical subspaces for the $G$-action with respect to $g$, and moreover, it acts as the identity 
	on the horizontal subspace (see Proposition 6.3 in \cite{AlexandrinoBettiol15}). Letting $\mathrm{Gr}_2(TM)$ denote the Grassmannian bundle of planes on $M$, 
	the Cheeger reparametrization is defined as the bundle automorphism $C_l:\mathrm{Gr}_2(TM)\to \mathrm{Gr}_2(TM)$ given by $C_l(P)=\Span\{C_l(x),C_l(y)\}$, 
	where $P=\Span\{x, y\}$. 
	
	\begin{lemma}\label{L:cheegerorthogonal}
		Let $(M,g_l)$ denote a Cheeger deformation of $(M,g)$ with respect to a $G$-action. Assume $H$ acts on $M$ by isometries of both $g$ and $g_l$.
		\begin{enumerate}
			\item If $p\in M$ and $x\in T_pM$ is orthogonal to both $G\cdot p$ and the $H\cdot p$ with respect to $g$, then $x$ is orthogonal to both orbits with respect to $g_l$.
			\item If $\gamma$ is a geodesic under $g$ that is orthogonal to both $G\cdot\gamma(0)$ and $H\cdot\gamma(0)$ with respect to $g$, then $\gamma$ is a geodesic 
			under $g_l$ that is orthogonal to all the $G$-orbits and the $H$-orbits that it meets with respect to $g_l$.
		\end{enumerate}
	\end{lemma}
	
	\begin{proof}
		The first statement follows from Equation \eqref{eq:cheeger} and the fact that $C_l(x) = x$ if $x$ is orthogonal to $G\cdot p$ with respect to $g$.
		For the second statement, if $\gamma$ is a geodesic under $g$ that is orthogonal to $G\cdot\gamma(0)$, then it is orthogonal to all the $G$-orbits it meets
		with respect to $g$. Thus the lift $\bar{\gamma}(t) = (e,\gamma(t))$ of $\gamma$ to $G \times M$ via $q$ is a geodesic under $\bar{g}_l$ 
		that is orthogonal to the $G$-orbits for the action given in Equation \eqref{eq:cheegeraction}. Since Riemannian submersions send horizontal geodesics to geodesics, 
		$\gamma$ is a geodesic under $g_l$. The first statement then implies that $\gamma$ is orthogonal to both $G\cdot\gamma(0)$ and $H\cdot\gamma(0)$, 
		and hence to all the $G$-orbits and the $H$-orbits that it meets, with respect to $g_l$.
	\end{proof}
	
	Next, we collect some tools for tracking the effect of Cheeger deformations on sectional curvatures. 
	The first result, which follows from work by M\"uter \cite{Muter87}, provides a condition under which one may reduce the number of zero curvature planes by applying the Cheeger deformation under an isometric $\SO(3)$ or $\SU(2)$-action (see also \cite[Corollary 2.2]{Bettiol14}):
	
	\begin{lemma}\label{L:planeprinciple}
		Suppose $G$ is a compact Lie group which acts isometrically on a non-negatively curved Riemannian manifold $(M,g)$, and let $(M,g_l)$ be a Cheeger deformation 
		via the $G$-action. Then $(M,g_l)$ has non-negative sectional curvature as well. Moreover, if $G=S^3$ or $G=\SO(3)$, then the following hold:
		\begin{enumerate}\setlength\itemsep{1pt}
			\item If a plane $P$ is positively curved with respect to $g$, then $C_l(P)$ is positively curved with respect to $g_l$.
			\item If the orthogonal projection of $P$ onto the $G$-orbit with respect to $g$ is two-dimensional, then $C_l(P)$ is positively curved with respect to $g_l$.
		\end{enumerate}
	\end{lemma}
	
	We will also need the following result, which combined with Lemma \ref{L:planeprinciple}, enables us to verify when a Cheeger deformed metric has $\pr{k}$.
	
	\begin{lemma}\label{L:cheegerskip}
		Assume $(M,g_l)$ has non-negative sectional curvature. Let $\mathcal{Z}$ denote the set of planes in a subspace $V$ of $T_pM$ that have zero sectional curvature 
		with respect to $g_l$, let $\cP$ denote the preimage of $\mathcal{Z}$ under the Cheeger reparametrization $C_l:\Gr_2(T_pM) \to \Gr_2(T_pM)$, 
		and let $k \in \{1,\dots,\dim M - 1\}$. 
		
		If $\cP$ has the property that there do not exist linearly independent vectors $x_0,x_1,\dots,x_k$ such that $\Span\{x_0,x_i\} \in \cP$ for all $i$, then  $(M,g_l)$ has $\Ric_k > 0$ on the subspace $V$.
	\end{lemma}
	
	\begin{proof}
		Assume there exists orthonormal vectors $z_0,z_1,\dots,z_m \in V$ such that $\sum_{i=1}^m \sec_{g_l}(z_0,z_i) \leq 0$. Since $\sec_{g_l} \geq 0$, 
		it follows that $\Span\{z_0,z_i\} \in \mathcal{Z}$ for all $i\geq 1$. Define $x_i = C_l^{-1}(z_i)$ for all $i$. Then $x_0,x_1,\dots,x_m$ are linearly independent 
		and $\Span\{x_0,x_i\} \in \cP$. 
		Therefore, $m\leq k-1$ by the assumption, and hence $\sum_{i=1}^k \sec_{g_l}(y_0,y_i) > 0$ for all orthonormal vectors 
		$y_0,y_1,\dots,y_k$.
	\end{proof}

\section{Constructions}\label{S:constructions}
	
	The main goal of this section is to prove that $S^3 \times \CP^2$ admits a metric with $\pr{4}$ as in Theorem \ref{main-thm:S3CP2}.
	We begin with a description of the metric with $\pr{2}$ on $S^3 \times S^3$ referenced in Example \ref{ex:s3xs3}.
	Wilking described the zero curvature planes for this metric in \cite[Section 5]{Wilking02}, from which it can be inferred that this metric has $\pr{2}$. 
	The second author with Dom\'inguez-V\'azquez and Gonz\'alez-\'Alvaro 
	extended this construction to products of the form $G\times G$, where $G$ is a semisimple Lie group (see \cite[Theorem E]{DVGAM23}).
	We describe the construction here to provide a simple example of how the tools from Section \ref{SS:Cheeger} are used.
	
	\begin{example}\label{ex:s3xs3revisited}
		Consider $S^3$ with the standard round metric $g_{S^3}$, and equip $S^3 \times S^3$ with the product metric $g \defeq g_{S^3}\oplus g_{S^3}$.
		Let $g_l$ denote the Cheeger deformation of $g$ with respect to the $S^3$-action given by
		\[
			a \cdot (p,q) = (ap,aq).
		\]
		It follows from Lemma \ref{L:planeprinciple} that $g_l$ has $\sec\geq 0$, and moreover, that the Cheeger reparametrization of a plane $P$ has zero curvature
		only if $P = \Span\{(x,0),(0,x)\}$ for some $x\neq 0$. Therefore, by Lemma \ref{L:cheegerskip}, we get that $(S^3 \times S^3,g_l)$ has $\pr{2}$.
		Lemma \ref{lem:isometry} then implies that this metric is invariant under the $(S^3 \times S^3 \times S^3)$-action defined in Example \ref{ex:s3xs3}.
	\end{example}
	
	Now, we will construct the metric referenced in Theorem \ref{main-thm:S3CP2}.
	Let $\phi: S^3 \to \SO(3)$ denote the standard covering map, and think of $\SO(3)$ as a subgroup of $\SU(3)$. Then $S^3\times\C\mathrm{P}^2$ admits a cohomogeneity one action 
	of $S^3\times S^3$ given by 
	\begin{equation}\label{eq:s3xcp2action}
		(a,b) * (p,[z]) = (apb^{-1},[\phi(b) z]).
	\end{equation}
	Consider the lift of this action to $S^3 \times S^5 \subset \HH \oplus \C^3$, which is given by\vspace{-0.08cm}
	\begin{equation}\label{eq:s3xs5}
	(a,b) * (p,z) = (apb^{-1},\phi(b)z).
	\end{equation}
	In what follows, we will refer to the action of the subgroup $1 \times S^3 \subset S^3\times S^3$ via $*$ simply as {\it the $S^3$-action}.
	We will refer to action of $S^1$ on $S^3 \times S^5$ by
	\[
	e^{i\theta} \star (p,z) = (p,e^{i\theta} z)
	\]
	simply as {\it the Hopf action}.
	
	Consider $S^3 \times S^5$ equipped with the standard product metric $g = g_{S^3} \oplus g_{S^5}$. 
	Then the $(S^3\times S^3)$-action from \eqref{eq:s3xs5} is by isometries of $g$.
	Let $g_l$ denote the Cheeger deformation of $g$ via the $S^3$-action, as defined in Section \ref{SS:Cheeger}.
	Note that because the Hopf action commutes with the $S^3$-action, $g_l$ is invariant under the Hopf action by Lemma \ref{lem:isometry}.
	Therefore, there exists a Riemannian metric $\check{g}_l$ on $S^3 \times \CP^2$ such that the quotient map 
	\begin{equation}\label{eq:quotient}
		(S^3 \times S^5,g_l) \to (S^3 \times \CP^2,\check{g}_l)
	\end{equation}
	is a Riemannian submersion.

	The space $S^3 \times \CP^2$ with the action in \eqref{eq:s3xcp2action} belongs to the family $P_A^7$ described in Proposition \ref{P:P_A^7} with $p_-=p_+=q_-=q_+=1$ (see \cite[p.174]{Hoelscher10}), 
	where we show that these manifolds cannot admit an invariant metric with $\pr{3}$.
	Hence the following result is optimal:
%	\todo[inline]{What other circle quotients do we get?}
%	and thus it cannot  
%	The next result shows that this cohomogeneity one manifold admits an invariant metric with $\pr{4}$. Moreover, this construction is optimal by Proposition \ref{P:P_A^7}.
	
	\begin{proposition}\label{P:s3xcp2}
		The metric $\check{g}_l$ from \eqref{eq:quotient} on $S^3 \times \CP^2$ has $\Ric_4 > 0$, and it is invariant under the cohomogeneity one $(S^3 \times S^3)$-action given by \eqref{eq:s3xcp2action}.
	\end{proposition}
	
	\begin{proof}
%		 We claim that $\check{g}_l$ has $\Ric_4 > 0$.
		
		First, we describe the tangent spaces of the orbits for the actions under consideration.
		Let $\Phi:\fsp(1) \to \Gamma (TS^5)$ denote the action field map induced by the $S^3 (\cong \Sp(1))$-action on $S^5$.
		Then the tangent spaces to the orbits of the $(S^3 \times S^3)$-action, the $S^3$-action, and the Hopf action on $S^3 \times S^5$ described above are given by 
		\begin{equation}\label{eq:s3xs3orbit}
			T_{(p,z)} ((S^3 \times S^3) * (p,z)) = \{(up-pv,\Phi(v)|_z):v\in\fsp(1)\},
		\end{equation}
		\begin{equation}\label{eq:s3orbit}
			T_{(p,z)} ((1 \times S^3) * (p,z)) = \{(-pv,\Phi(v)|_z):v\in\fsp(1)\},
		\end{equation}
		\begin{equation}\label{eq:s1orbit}
			T_{(p,z)} (S^1 \star (p,z)) = \Span_\RR\{(0,iz)\}.
		\end{equation}
		
		The metric $\check{g}_l$ on $S^3 \times \CP^2$ inherited from $(S^3 \times S^5,g_l)$ will be invariant under the cohomogeneity one $(S^3\times S^3)$-action 
		\eqref{eq:s3xcp2action} by Lemma \ref{lem:isometry}.
		Thus it suffices to check that $\Ric_4 > 0$ holds on a horizontal geodesic connecting the singular orbits in $(S^3 \times \CP^2,\check{g}_l)$.
		For our purposes, we will consider the curve $\gamma:[0,\pi/2] \to S^3 \times S^5$ given by $\gamma(t) = (1,z(t))$, where $z(t) = (\cos t, i \sin t, 0) \in \C^3$.
		Because $\gamma$ is initially orthogonal to the $S^1$-, $S^3$-, and $(S^3\times S^3)$-orbits with respect to the product metric $g$, applying Lemma \ref{L:cheegerorthogonal}, 
		we get that $\gamma$ is a geodesic with respect to $g_l$ that is orthogonal to all of the orbits that it meets for these actions.
		Therefore, the induced curve on $S^3 \times \CP^2$ is also a geodesic with respect to $\check{g}_l$ that is orthogonal to all of the $(S^3 \times S^3)$-orbits that it meets, 
		meaning it is horizontal for the action. Given $t\in[0,\pi/2]$, let $\cH|_t \subset T_{\gamma(t)}(S^3 \times S^5)$ denote the horizontal space for the Riemannian submersion 
		$(S^3 \times S^5,g_l) \to (S^3 \times \CP^2,\check{g}_l)$ at the point $\gamma(t) = (1,z(t))$. 
		Similar to Lemma \ref{lem:oneill}, by the Gray-O'Neill formula, it is enough to check that $g_l$ has $\pr{4}$ on $\cH|_t$ 
		for all $t\in[0,\pi/2]$.
		
		Consider the Cheeger reparametrization $C_l:\Gr_2(T_{\gamma(t)}(S^3 \times S^5))\to \Gr_2(T_{\gamma(t)}(S^3 \times S^5))$ described in Section \ref{SS:Cheeger}.
		By Lemma \ref{L:planeprinciple}, $g_l$ has non-negative curvature. 
		Furthermore, if a plane $C_l(P) \subset \cH|_t$ has zero curvature with respect to $g_l$, 
		then $P$ has zero curvature with respect to the product metric $g$, and the projection of $P$ onto the $S^3$-orbit with respect to $g$ has dimension at most one.
		Combining these observations with Equations (\ref*{eq:s3orbit}) and (\ref*{eq:s1orbit}) and the fact that the factors of $(S^3 \times S^5,g)$ have positive sectional curvature, 
		it is straightforward to justify the following:
		
		\noindent If a plane $C_l(P) \subset \cH|_t$ has zero curvature with respect to $g_l$, then there exist $x\in\fsp(1)$, $y\in T_{z(t)}S^5$, and $\lambda \in \RR$ such that
		\begin{enumerate}[label=(\arabic{enumi})]
			\item $P = \Span\{(x,0),(0,y)\}$,\label{item:first}
			\item For all $v\in \fsp(1)$, one has $g_{S^5}(y^\top,\Phi(v)) = \lambda \cdot g_{S^3}(x,v)$, where $y^\top$ denotes the part of $y$ tangent to the $S^3$-orbit in $S^5$. 
			\label{item:middle}
			\item $g_{S^5}(y,iz(t)) = 0$.\label{item:last}
		\end{enumerate}
		
		Let $\cP$ denote the set of all planes $P$ such that $C_l(P) \subset \cH|_t$  and $C_l(P)$ has zero curvature with respect to $g_l$.
		Assume that $(x_0,y_0)$, $(x_1,y_1)$, $\dots$, $(x_n,y_n) \in C_l^{-1}(\cH|_t)$ are linearly independent, and moreover, that the planes $\Span\{(x_0,y_0),(x_m,y_m)\}$ 
		are elements of $\cP$ for $m=1,\dots,n$. By Lemma \ref{L:cheegerskip}, it suffices to show that $n\leq 3$ for all $t\in [0,\pi/2]$.
		
		First suppose $t\in(0,\frac \pi 2 )$, so that $\Phi|_{z(t)}:\fsp(1)\to T_{z(t)}S^5$ has trivial kernel.
		It is straightforward to check that the vectors $\Phi(i),\Phi(j),\Phi(k)$ are orthogonal.
		Define the injective linear transformation $\Psi_t:\fsp(1)\to T_{z(t)}S^5$ that sends the $g_{S^3}$-orthonormal basis vectors $i,j,k$ respectively to the $g_{S^5}$-orthogonal 
		vectors $\Phi(i)/|\Phi(i)|^2,\Phi(j)/|\Phi(j)|^2,\Phi(k)/|\Phi(k)|^2$. Consequently, it follows from condition \ref*{item:middle} that $y^\top$ is proportional to $\Psi_t(x)$.
		Thus if $P\in \cP$, then $P$ must be of the form
		\[
			P = \Span\{(x,0),(0,\lambda \cdot \Psi_t(x)+y^\perp)\}
		\]
		for some constant $\lambda \in \RR$ and some vector $y^\perp$ orthogonal to the $S^3$-orbit in $S^5$.
		If $y_0$ is non-zero, then each $y_m$ is proportional to $y_0$, which implies that the vectors $x_1,\dots,x_n$ are linearly independent.
		But since $x_1,\dots,x_n \in \fsp(1)$, we have $n\leq 3$ in this case. 
		Suppose instead $y_0 = 0$, then $x_m$ is proportional to $x_0$ for all $m$, while each $y_m = \lambda_m \Psi_t(x_0) + y_m^\perp$.
		Because of the restriction on $x_1,\dots,x_n$, we have that $y_1,\dots,y_n$ must be linearly independent.
		Note that along the curve $z(t)$ in $S^5$ for $t\in[0,\pi/2]$, each $y_m^\perp$ must be orthogonal to the $S^1$-orbit in $S^5$ (by condition \ref*{item:last}) 
		and the $S^3$-orbit in $S^5$ with respect to $g_{S^5}$.  It follows that each $y_m^\perp$ must be an element of $\Span_\RR\{(1,-i,0),(i,1,0)\}$ if $t=\pi/4$ 
		or $\Span_\RR\{(\sin t, -i \cos t,0)\}$ if $t\in(0,\frac \pi 4 )\cup(\frac \pi 4 ,\frac \pi 2 )$.
		Thus $\Span\{y_1^\perp,\dots,y_n^\perp\}$ is at most $2$-dimensional, $\Span\{y_1,\dots,y_n\}$ is at most $3$-dimensional, and hence $n \leq 3$ for all $t\in(0,\frac \pi 2 )$.
		
		Now suppose $t=0$, which means $\Phi(i) = 0$.
		Notice $\{\Phi(j)/2,\Phi(k)/2\}$ are orthonormal in $T_{z(0)} S^5$.
		Thus it follows from condition \ref*{item:middle} that if $y^\top \neq 0$, then $g_{S^3}(x,i)=0$ and $y^\top$ is proportional to $\Phi(x)$.
		Hence, if $P\in \cP$, then $P$ must be of the form
		\begin{equation}\label{eq:P}
			P = \Span\{(x,0),(0,\lambda \cdot \Phi(x)+y^\perp)\}
		\end{equation}
		for some constant $\lambda \in \RR$ and vector $y^\perp$ orthogonal to the $S^3$-orbit in $S^5$, where if $\lambda \neq 0$, then $x\in\Span_\RR\{j,k\}$.
		In this case, the intersection of the orthogonal complements of the $S^1$-orbit and the $S^3$-orbit in $S^5$ is given by $\Span_\RR\{(1,0,0),(0,0,i)\}$.
		Then similar to the previous case, it follows that $n \leq 3$ when $t=0$.
		
		Finally, suppose $t=\pi/2$, which means $\Phi(j) = 0$.
		Then $\{\Phi(i)/2,\Phi(k)/2\}$ are orthonormal in $T_{z(\pi/2)} S^5$, and it follows from condition \ref*{item:middle} that every $P\in \cP$ must be of the form (\ref*{eq:P}) 
		for some constant $\lambda \in \RR$ and vector $y^\perp$ orthogonal to the $S^3$-orbit in $S^5$, where if $\lambda \neq 0$, then $x\in\Span_\RR\{i,k\}$.
		In this case, the intersection of the orthogonal complements of the $S^1$-orbit and the $S^3$-orbit in $S^5$ is $\Span_\RR\{(1,0,0),(0,0,1)\}$.
		Thus similar to the previous cases, it follows that $n\leq 3$ when $t=\pi/2$.
		
		Since $n \leq 3$ for all $t\in [0,\frac \pi 2 ]$, we have proved that $(S^3 \times \CP^2,\check{g}_l)$ has $\pr{4}$.
	\end{proof}
	
\section{Obstructions}\label{S:obstructions}
	
	In this section, we investigate cohomogeneity one manifolds in dimensions $5$--$7$, and show that some of them cannot admit invariant metrics with $\pr{k}$ 
	for certain values of $k$.
	We are motivated by the following example:
%	Verdiani and Ziller investigated the following example in \cite{VerdianiZiller14}:
	
	\begin{example}\label{ex:PD7}
		The family $P_D^7$ has group diagrams of the form
		\[
		\Z_n \subset \{(e^{ip\theta},e^{iq\theta})\}, \Delta S^3 \cdot \Z_n \subset S^3 \times S^3,
		\]
		where $n = 2$ and $p$ or $q$ is even, or $n = 1$ and $p$ and $q$ are arbitrary.
		If $n = 2$ and $(p,q) = (p,p+1)$, the corresponding manifolds are Eschenburg spaces, which are known to admit invariant metrics of positive sectional curvature.
		If $n = 1$ and $p = q = 1$, the corresponding manifold is diffeomorphic to $S^3 \times \CP^2$, which does not admit an invariant metric with $\pr{2}$ by Verdiani and Ziller \cite{VerdianiZiller14}.
		Furthermore, if $n = 1$ and $(p,q) \neq (1,1)$, they showed that the corresponding manifolds do not admit invariant metrics with $\pr{3}$. 
	\end{example}
	
%	We will be following similar techniques, but have chosen to employ Lemma \ref{L:obstruction} in the present article.
	
\subsection{Dimension 5}
	
	The Brieskorn variety $M_d^{2n-1}$ is defined by 
	\[
		M_d^{2n-1}=\{(z_0,\ldots, z_n)\in\C^{n+1}: z_0^d+z_1^2+z_n^2=0, |z_0|^2+\ldots+|z_n|^2=1\}.
	\]
	These are homeomorphic to spheres if both $n$ and $d$ are odd, and are exotic spheres if $2n-1\equiv 1\mod 8$.
	Furthermore, $M_2^{2n-1}$ is diffeomorphic to the unit tangent bundle of $S^n$ (see \cite[p. 161]{GVWZ06}).
	Calabi proved that the five-dimensional Brieskorn variety admits a cohomogeneity one action by $\mathrm{SO}(2)\mathrm{SO(3)}$. 
	This result was later generalized to $M_d^{2n-1}$ by Hsiang and Hsiang \cite{HsiangHsiang89}. The first result of this section rules out positive $\Ric_2$ 
	for the five-dimensional Brieskorn varieties that are not homeomorphic to spheres. %a specific class of five-dimensional Brieskorn varieties.
	
%	\todo[inline]{According to Hoelscher '10 and Kwon and Van Koert '16 (https://arxiv.org/pdf/1310.0343), the manifolds in this next proposition are spin and hence are $S^2 \times S^3$.}
	
	\begin{proposition}\label{P:M_d^5}
		If $d$ is even, then the Brieskorn variety $M_d^5$ with the group diagram
		\[
			\{(1,1)\}\subset \{(e^{i\theta}, 1)\}, \{(e^{jd\theta}, e^{i\theta})\}\subset S^3\times S^1
		\]                
		does not admit an $(S^3\times S^1)$-invariant metric with positive $2$-intermediate Ricci curvature.
	\end{proposition}

%	Compare this result to the metric with $\pr{2}$ on $S^2 \times S^3$ described in Example \ref{ex:s3xs3}.
%	Following the notation from that example, this metric is invariant under the cohomogeneity one action by the group $\{1\} \times S^3 \times S^1$.
%	In contrast Proposition \ref*{P:M_d^5}, the group diagram for this example is of the form
%	\[
%		\{(1,1)\}\subset \{(e^{i\theta}, e^{i\theta})\}, \{(e^{i\theta}, e^{i\theta})\} \subset S^3\times S^1.
%	\]
	
%	\todo[inline]{Do these manifolds admit $\Ric_3 > 0$.}
	
	\begin{proof}
		Following the discussion of Section \ref{SS:cohomogeneity one}, we consider an arbitrary unit speed horizontal geodesic $\gamma:(0,L)\to M_d^5$ and study the family $\Lambda$ 
		of Jacobi fields along $\gamma$ generated by the $(S^3 \times S^1)$-action. Since the principal isotropy group is zero-dimensional, no Jacobi field $J \in \Lambda$ 
		vanishes on the interior $(0, L)$. Therefore,
		\[
			\{J \in \Lambda: J(t)=0 \text{ for some }t\} = \{J\in\Lambda:J(nL)=0\text{ for some }n\in\mathbb{Z}\}.
		\]
		Moreover, the Jacobi fields $J$ satisfying $J(nL) = 0$ are the Killing fields associated with elements of the Lie algebra of the isotropy group $G_{\gamma(nL)}$.
		
	         Note that by Lemma \ref{L:diagram}, any other group diagram representing a manifold equivariantly diffeomorphic to $M_d^5$ is of the form (up to switching $K_-$ and $K_+$):
		\[
			\{(1,1)\}\subset \{(e^{x\theta}, 1)\}, \{(e^{yd\theta},e^{i\theta})\}\subset S^3\times S^1,\vspace{0.1cm}
		\]
		where $x, y\in\im(\mathbb{H})\cap S^3$. Furthermore, for all these group diagrams, the generators $w_-$ and $w_+$ of the Weyl group are given by $w_-=(-1,1)$ 
		and $w_+=(1,-1)$. In particular, the Weyl group has order four, and the horizontal geodesic $\gamma$ has length $4L$. 
		By Remark \ref{R:Weyl}, it follows that $G_{\gamma(0)}=G_{\gamma(2L)}=G_{\gamma(4L)}$ and $G_{\gamma(L)}=G_{\gamma(3L)}$.
		Since these groups are one-dimensional and their Jacobi fields on $M_d^5$ are linearly independent, the space $\Span\{J\in\Lambda\mid J(t)=0~{\text{for some}}~t\}$ is two-dimensional. 
		Thus by Lemma \ref{L:obstruction}, we conclude that $M_d^5$ cannot admit an invariant metric with $\pr{2}$.
	\end{proof}
	
	\begin{proposition}\label{P:QA5}
		The cohomogeneity one manifold $Q_A^5$ with the group diagram 
		\[
		\Z_n\subset\{(e^{ip\theta},e^{i\theta})\}, \{(e^{ip\theta},e^{i\theta})\}\subset S^3\times S^1
		\]
		does not admit any $(S^3\times S^1)$-invariant metric with positive $2$-intermediate Ricci curvature if $n=1$.
	\end{proposition}
	
	\begin{proof}
		The proof follows using the same argument applied for Proposition \ref{P:M_d^5}. The only difference is that here all possible group diagrams are of the form 
		(up to switching $K_-$ and $K_+$):
		\[
			\{(1,1)\}\subset \{(e^{xp\theta},e^{i\theta})\}, \{(e^{yp\theta},e^{i\theta})\}\subset S^3\times S^1,\vspace{0.1cm}
		\]
		where $x, y\in\im(\mathbb{H})\cap S^3$. Moreover, the generators of the Weyl group are given by $w_-=w_+=(1,-1)$ or $(-1,-1)$, depending on whether $p$ is even or odd,
		and thus the geodesic $\gamma$ has length $2L$. Therefore, the space $\Span\{J\in\Lambda\mid J(t)=0~{\text{for some}}~t\}$ is two-dimensional if $x\neq y$ 
		and one-dimensional otherwise.	
		%Here, the generators $w_-$ and $w_+$ of the Weyl group either both equal $(1,-1)$ if $p$ is even or both equal $(-1,-1)$ if $p$ is odd. 
		%Therefore, the Weyl group has order two, the horizontal geodesic $\gamma$ has length $2L$, and we have:\vspace{0.1cm}
		%$$G_{\gamma(0)}=G_{\gamma(L)}=G_{\gamma(2L)}=\{(e^{ip\theta},e^{i\theta})\}.$$
		%Hence the space $\Span\{J\in\Lambda\mid J(t)=0~{\text{for some}}~t\}$ is one-dimensional and thus the metric giving rise to the above group diagram cannot have positive $\Ric_3$. 
		%Moreover, any other $G$-invariant metric on $Q_A^5$ gives rise to a group diagram of the form (up to switching $K_-$ and $K_+$):\vspace{-0.1cm}
		%\[
		%	\{(1,1)\}\subset \{e^{xp\theta},e^{i\theta}\}, \{e^{yp\theta},e^{i\theta}\}\subset S^3\times S^1,\vspace{0.1cm}
		%\]
		%where $x, y\in\im(\mathbb{H})\cap S^3$. For this group diagram, still $w_-$ and $w_+$ both equal either $(1,-1)$ or $(-1,-1)$, and the geodesic $\gamma$ still has length $2L$. 
		%However, the isotropy groups $G_{\gamma(0)}$ and $G_{\gamma(L)}$ might not be the same, in which case, they give rise to two linearly independent Jacobi fields 
		%vanishing at $\gamma(0)$ and $\gamma(L)$, respectively. 
		%Therefore, the space $\Span\{J\in\Lambda\mid J(t)=0~{\text{for some}}~t\}$ 
		%would be two-dimensional and thus this new metric may possibly have positive $3$-intermediate Ricci curvature, but it cannot have positive $2$-intermediate Ricci curvature.  
		%Altogether, we conclude that $Q_A^5$ does not support any $(S^3\times S^1)$-invariant metric with positive $\Ric_2$.
	\end{proof}
	
	The manifolds in the family $Q_A^5$ are diffeomorphic to $S^3 \times S^2$ (see \cite[Table VII]{Hoelscher10}).
	In the next example, we show that $S^3 \times S^2$ with the action described in Example \ref{ex:s3xs3} is a member of $Q_A^5$, and the fact that it admits an invariant metric with $\pr{2}$ shows that Proposition \ref*{P:QA5} does not generalize to the case in which the principal isotropy group has order two.
	
	\begin{example}
%		We compare Proposition \ref*{P:QA5} to the metric with $\pr{2}$ on $S^3 \times S^2$ described in Example \ref{ex:s3xs3}.
		Identify $S^2$ with $S^3/S^1$, and consider the cohomogeneity one action of $S^3 \times S^1$ on $S^3 \times S^2$ given by\vspace{-0.1cm}
		\[
			(a,e^{i\theta})\cdot(p,q S^1) = (a p e^{-i \theta} , a q S^1).
		\]
		The metric from Example \ref{ex:s3xs3} is invariant under this action, and the curve $\gamma(t) = (e^{jt}, e^{-jt} S^1)$ is a horizontal geodesic. 
		Along $\gamma$, the singular orbits occur at $\gamma(\pi m/4)$ for any integer $m$.
		The isotropy groups at $\gamma(0)$ and $\gamma(\pi/4)$ are given by $\{(e^{i\theta},e^{i\theta})\}$ and $\{(e^{k\theta},e^{-i\theta})\}$, respectively.
		Moreover, the principal isotropy group is given by $H=\{\pm(1,1)\}$, and the Weyl group is generated by $w_-=(i,i)\cdot H$ and $w_+=(k,-i)\cdot H$. 
		It follows from Lemma \ref{L:diagram} that $S^3 \times S^2$ with this action is equivariantly diffeomorphic to a manifold in the family $Q_A^5$ described in Proposition \ref{P:QA5}, but with $n=2$ and $p=1$. 
%		With the metric inherited from the standard product metric on $S^3 \times S^3$, this manifold has $\Ric_4 > 0$.
		Note that the isotropy group at $\gamma(\pi/2)$ is $\{(e^{-i\theta},e^{i\theta}):\theta \in \R\}$.
		Therefore, unlike the case in Proposition \ref{P:QA5}, there are enough Jacobi fields that vanish along $\gamma$ for $\Ric_2 > 0$ to be permissible. 
%		In fact, the metric on $S^3 \times S^2$ with $\pr{2}$ described in Example \ref{ex:s3xs3} is invariant under this action.
		
%		\todo[inline]{Should also discuss $S^1 \times S^3$ action. Geodesic $\gamma(t) = (1,e^{jt})$. Principal isotropy $\pm(1,1)$. Isotropy at $t=0$ is $(e^{i\theta},e^{i\theta})$ and at $t=\pi/2$ is $(e^{i\theta},e^{-i\theta})$.}
	\end{example}
	
\subsection{Dimension 6}
	
	Consider the cohomogeneity one $(S^3\times S^3)$-action on $S^3\times S^3$ given by
	\[
		(a,b)\cdot(p,q) = (a * p,bq),
	\]
	where $*$ represents the suspension of the action of $S^3$ on $S^2$.
	The geodesic $\gamma(t) = (e^{it},1)$ is orthogonal to the orbits of this action, and the associated group diagram is
	\[
		S^1\times 1\subset S^3\times 1, S^3\times 1\subset S^3\times S^3.
	\]
	Furthermore, the standard product metric on $S^3 \times S^3$ has $\Ric_4 > 0$ and is invariant under the given action. In the following proposition, we prove that this cannot be 
	promoted to a $(S^3\times S^3)$-invariant metric with $\Ric_3>0$.
	
	\begin{proposition}
		The cohomogeneity one manifold $S^3\times S^3$ with the group diagram
		\[
		S^1\times 1\subset S^3\times 1, S^3\times 1\subset S^3\times S^3,
		\]
		where $S^1$ is the complex circle $\{e^{i\theta}\}$, does not admit any $(S^3\times S^3)$-invariant metric with positive $3$-intermediate Ricci curvature.
	\end{proposition}
	
	\begin{proof}
		In this case, the principal isotropy group is one-dimensional, and since the associated Jacobi fields along $\gamma$ are trivial, they do not contribute to the dimension 
		of $\Span\{J\in\Lambda\mid J(t)=0~{\text{for some}}~t\}$.
		Therefore, we only need to focus on Jacobi fields induced by the singular isotropy groups that do not vanish at regular points.
		
		Note that any other cohomogeneity one manifold which is equivariantly diffeomorphic to $S^3\times S^3$ has the same group diagram as above 
		(up to replacing the complex circle $\{e^{i\theta}\}$ with another circle subgroup of $S^3$). Moreover, for the given group diagram, we have $w_-=w_+=(j,1)\cdot H$,
		and hence the length of the geodesic $\gamma$ equals $2L$. Therefore, 
		$$\Span\{J\in\Lambda\mid J(t)=0~{\text{for some}}~t\}=\Span\{X_2, X_3\},$$
		where $X_2$ and $X_3$ denote the Killing fields on $S^3 \times S^3$ corresponding to $j$ and $k$ in the Lie algebra of the first $S^3$ factor. This completes the proof.
	\end{proof}
	
\subsection{Dimension 7}
	In this final section, we first prove the obstruction statement in Theorem \ref{main-thm:S3CP2}, then we complete the proofs of the remaining cases in Theorem \ref{main-thm:obstruction}.
	
	\begin{proposition}\label{P:P_A^7}
		The cohomogeneity one manifold $P_A^7$ with the group diagram
		\[
			H=\langle(i,i)\rangle\subset\{(e^{ip_-\theta},e^{iq_-\theta})\}, \{(e^{jp_+\theta},e^{jq_+\theta})\}\cdot H\subset S^3\times S^3,~~p_-, q_-\equiv 1\mod 4
		\]
		does not admit any $(S^3\times S^3)$-invariant metric with positive $3$-intermediate Ricci curvature.
%		provided that one of the following conditions holds:
%		\begin{enumerate}
%			\item $p_-=p_+=q_-=q_+=1$.
%			\item One of $p_+$ and $q_+$ is even and the other one is odd.
%		\end{enumerate}
	\end{proposition}
	
	\begin{proof} 
		Note that $N(H) = \{(e^{i\theta},e^{i\theta})\} \cdot \langle(j,j)\rangle$, and hence $N(H)_0 = \{(e^{i\theta},e^{i\theta})\}$. 
		Letting $K_- = \{(e^{ip_-\theta},e^{iq_-\theta})\}$ and $K_+ = \{(e^{jp_+\theta},e^{jq_+\theta})\}\cdot H$, for any $n \in N(H)_0$, we have $n K_- n^{-1} = K_-$ and $n K_+ n^{-1} = \{(e^{l p_+ \theta},e^{l q_+ \theta})\}\cdot H$ for some $l = (\cos \phi) j + (\sin \phi) k$.
		Hence any group diagram equivalent to $H \subset K_-, K_+ \subset G$ is of the form (up to switching $K_-$ and $K_+$):
		\[
			H'=\langle(u,x)\rangle  \subset  \{(e^{u p_- \theta},e^{x q_- \theta})\} ,  \{(e^{v p_+ \theta},e^{y q_+ \theta})\} \cdot H'  \subset S^3 \times S^3
		\]
		for some $u,v,x,y \in\Im(\mathbb{H})\cap S^3$ such that $u$ is orthogonal to $v$ and $x$ is orthogonal to $y$.
		Since $p_-, q_-\equiv 1\mod 4$, we have $w_- = (e^{u p_- \frac{\pi}{4}}, e^{x q_- \frac{\pi}{4}}) \cdot H'$.
		We may assume $\gcd(p_+,q_+)=1$, in which case
		\[
			w_+ = 
			\begin{cases}
				(-1,1) \cdot H' & \text{if } p_+ \not\equiv q_+ \mod 2, \\
				(v,y)\cdot H' & \text{if } p_+, q_+ \text{ are odd and } p_+ \equiv q_+ \mod 4, \\
				(-v,y)\cdot H' & \text{if } p_+, q_+ \text{ are odd and } p_+ \not\equiv q_+ \mod 4.
			\end{cases}
		\]
		In each case, the order of the Weyl group is four, and hence the length of the geodesic $\gamma$ equals $4L$.
		It follows that
		\[
			G_{\gamma(0)} = G_{\gamma(2L)} = G_{\gamma(4L)} = \{(e^{up_-\theta},e^{xq_-\theta})\}.
		\]
%		Moreover, denoting $w = uv$ and $z = xy$, we get that
%		\begin{align*}
%			G_{\gamma(0)} &= G_{\gamma(2L)} = G_{\gamma(4L)} = \{(e^{up_-\theta},e^{xq_-\theta})\},\\
%			G_{\gamma(L)} &= \{(e^{vp_+\theta},e^{yq_+\theta})\} \cdot H',\\
%			G_{\gamma(3L)} &= \{(e^{wp_+\theta},e^{zq_+\theta})\} \cdot H'.
%		\end{align*}
		Next, for representatives $u_- = (e^{u p_- \frac{\pi}{4}}, e^{x q_- \frac{\pi}{4}})$ and $u_+ \in\{(-1,1),(v,y),(-v,y)\}$ for $w_-$ and $w_+$ respectively, it follows from the assumption $p_-,q_- \equiv 1 \mod 4$ that 
		\[
			G_{\gamma(3L)} = (u_+u_-)G_{\gamma(L)}(u_+u_-)^{-1} = u_+ \of{\{(e^{uvp_+\theta},e^{xyq_+\theta})\} \cdot H'} (u_+)^{-1} = \{(e^{uvp_+\theta},e^{xyq_+\theta})\} \cdot H',
		\]
		Thus $G_{\gamma(L)} \neq  G_{\gamma(3L)}$, and it follows that the space spanned by the transverse Jacobi fields along $\gamma$ given by the $(S^3 \times S^3)$-action which vanish at some point is three-dimensional, thus proving the claim by Lemma \ref{L:obstruction}.
	\end{proof}

	\begin{proposition}
		The cohomogeneity one manifold $N_C^7$ with the group diagram 
		\[
			H\subset\{(e^{ip\theta},e^{iq\theta})\}, S^3\times\Z_n\subset S^3\times S^3,~{\text{where}}~(q,n)=1~{\text{and}}~\Z_n\cong H\subset\{(e^{ip\theta},e^{iq\theta})\}
		\]
		does not admit any $(S^3\times S^3)$-invariant metric with positive $2$-intermediate Ricci curvature if $n=1$ or $2$.
	\end{proposition}
	
	\begin{proof}
		Note that for $n\in\{1,2\}$, choosing a different $G$-invariant metric changes the group diagram to (up to switching $K_-$ and $K_+$):
		\begin{equation*}\label{eq:N_C^7}
			H\subset\{(e^{xp\theta},e^{yq\theta})\}, S^3\times\Z_n\subset S^3\times S^3\hspace{0.3cm}(\Z_n\cong H\subset\{(e^{xp\theta},e^{yq\theta})\}),		
		\end{equation*}
		where $x, y\in\im(\mathbb{H})\cap S^3$. For the case in which $n=1$, one has $w_+=(-1,1)$. Moreover, depending on the parity of $p$, $q$, $p/{\gcd(p,q)}$, 
		and $q/{\gcd(p,q)}$, $w_-$ is one of $(1,-1)$, $(-1,1)$, or $(-1,-1)$.
		%\begin{itemize}
		%	\item $p$ and $q$ are both odd. In this case, $w_-=(-1,-1)$ and $w_+=(-1,1)$. Moreover, the Weyl group has order four and hence the length of the geodesic $\gamma$ equals $4L$.
		%	\item $p$ is even and $q$ is odd. In this case, $w_-=(1,-1)$, $w_+=(-1,1)$, and $\gamma$ has length $4L$.
		%	\item $p$ is odd and $q$ is even. In this case, $w_-=(-1,1)$, $w_+=(-1,1)$, and $\gamma$ has length $2L$.
		%	\item $p$ and $q$ are both even. In this case, still $w_+=(-1,1)$ but $w_-$ is one of $(-1,-1)$, $(1,-1)$, or$(-1,1)$, depending on whether $p/{(p,q)}$ and $q/{(p,q)}$ 
		%	are both odd or one is odd and the other even. Hence the geodesic $\gamma$ might have length $2L$ or $4L$.
		%\end{itemize}
		In all of the cases, the geodesic $\gamma$ has length equal to either $2L$ or $4L$. Furthermore, there are two distinct singular isotropy groups along $\gamma$, 
		one of which has dimension one and the other one has dimension three. Thus the space $\Span\{J\in\Lambda\mid J(t)=0~{\text{for some}}~t\}$ is four-dimensional. 
		This completes the proof in the case where $n=1$.
		
		\vspace{0.1cm}
		
		Next, consider the case in which $n=2$. In this case, $q$ is odd. Moreover, if $p$ is odd (resp. even), then $H=\{\pm(1,1)\}$ (resp. $H=\langle(1,-1)\rangle$). 
		Furthermore, the generators of the Weyl group are given by $w_+=(-1,1)\cdot H$, and $w_-=(\pm x,y)\cdot H$ or $w_-=(\pm1,y)\cdot H$, depending on whether $p$ 
		is odd or even. Therefore, the geodesic $\gamma$ has length $4L$, $G_{\gamma(0)}=G_{\gamma(2L)}$, and $G_{\gamma(L)}=G_{\gamma(3L)}$, 
		from which the result follows.
	\end{proof}
	
	\begin{proposition}
		The cohomogeneity one manifold $Q_C^7$ with the group diagram 
		\[
			S^1\times 1\times 1\cdot\Z_n\subset\{(e^{i\phi},e^{ib\theta},e^{i\theta})\}, S^3\times 1\times 1\cdot\Z_n\subset S^3\times S^3\times S^1,
		\]
		where $\Z_n\subset\{(1,e^{ib\theta},e^{i\theta})\}$ and $H_0=\{(e^{i\phi},1,1)\}$, does not admit any $(S^3\times S^3\times S^1)$-invariant metric with positive $3$-intermediate  
		Ricci curvature if $n=1$.
	\end{proposition}
	
	\begin{proof}
		If $n=1$, then all possible group diagrams for the family $Q_C^7$ are of the form:\vspace{0.02cm}
		\[
			\{( e^{x\phi} , 1, 1)\}\subset\{( e^{x\phi} ,  e^{y b\theta}  , e^{i\theta})\}, S^3\times 1\times 1\subset S^3\times S^3\times S^1 \vspace{0.02cm}
		\]
		for some $x,y \in \Im(\HH) \cap S^3$. 
		For such a group diagram, one has $w_-=(1,\pm1, -1)\cdot H$ 
		and $w_+ =(z,1,1)\cdot H$ for some $z \in \Im(\HH) \cap S^3$ such that $z$ is orthogonal to $x$. Therefore, the geodesic $\gamma$ has length $4L$ and every singular isotropy group along $\gamma$ equals either $G_{\gamma(0)}$ 
		or $G_{\gamma(L)}$. Thus $\Span\{J\in\Lambda\mid J(t)=0~{\text{for some}}~t\}$ is three-dimensional and the claim follows.
	\end{proof}
	
	\begin{proposition}\label{prop:productactions4xs3}
		The seven-dimensional cohomogeneity one manifold given by the group diagram\vspace{-0.15cm}
		\[
			\{(1,1)\}\subset S^3\times 1, S^3\times 1\subset S^3\times S^3
		\]
		does not admit any $(S^3\times S^3)$-invariant metric with positive $3$-intermediate Ricci curvature.
	\end{proposition}
	
	\begin{proof}
		Since $w_-=w_+=(-1,1)$, the geodesic $\gamma$ has length $2L$ and thus the space $\Span\{J\in\Lambda\mid J(t)=0~{\text{for some}}~t\}$ has dimension three.
		Since in this case the group diagram is unique, the claim follows.
	\end{proof}
	
	\begin{remark}
	Thinking of $S^4 \subset \HH \oplus \RR$, the manifold in Proposition \ref{prop:productactions4xs3} is $S^4 \times S^3$ with the $(S^3\times S^3)$-action given by 
	$(a,b) \cdot ((x,t),y) = ((ax,t),by)$.
	\end{remark}

\bibliographystyle{alpha}
\bibliography{bibfile}

\end{document}